\documentclass{article}
\usepackage{arxiv}

\usepackage[utf8]{inputenc}
\usepackage[T1]{fontenc}
\usepackage{cite}
\usepackage{hyperref}
\hypersetup{
  colorlinks=true,
  linkcolor=blue,
  citecolor=blue,
  urlcolor=blue,
  pdftitle={Noncommutative Anisotropic Diffusion in Hilbert Space. II. Global Closure of the Logarithmic Gradient, Lower Bounds, and Nanosystem Applications},
  pdfauthor={E. Yu. Shchetinin and A. A. Shevchuk and S. I. Salpagarov}
}
\usepackage{url}
\usepackage{amsmath,amssymb,amsthm,amsfonts}
\usepackage{nicefrac}
\usepackage{microtype}
\usepackage{graphicx}

\newtheorem{theorem}{Theorem}[section]
\newtheorem{lemma}[theorem]{Lemma}
\newtheorem{proposition}[theorem]{Proposition}
\newtheorem{corollary}[theorem]{Corollary}
\newtheorem{assumption}[theorem]{Assumption}
\newtheorem{definition}[theorem]{Definition}
\newtheorem{remark}[theorem]{Remark}

\newcommand{\Hh}{\mathcal H}
\newcommand{\LL}{\mathcal L}

\newcommand{\LLtwo}{\mathcal L_2}
\newcommand{\E}{\mathbb E}

\newcommand{\Tr}{\operatorname{Tr}}
\newcommand{\Ent}{\mathrm H}
\newcommand{\Eval}{\mathcal L_{\mathrm{val}}}

\newcommand{\Asf}{\mathsf A}

\newcommand{\subjclass}[2][2020]{%
  \par\addvspace\medskipamount
  {\small\textit{2020 Mathematics Subject Classification.} #2\par}}

\title{Noncommutative Anisotropic Diffusion in Hilbert Space.\\
       II. Global Closure of the Logarithmic Gradient, Lower Bounds,
       and Nanosystem Applications}

\author{%
  E.\,Yu.~Shchetinin \\
  Sevastopol State University \\
  Sevastopol, Russia \\
  \texttt{riviera-molto@mail.ru}
  \And
  A.\,A.~Shevchuk \\
  Sevastopol State University \\
  Sevastopol, Russia \\
  \texttt{andreiluck11@yandex.ru}
  \AND
  S.\,I.~Salpagarov \\
  RUDN University \\
  Moscow, Russia \\
  \texttt{salpagarov-si@rudn.ru}
}

\begin{document}
\maketitle

% -- Abstract --------------------------------------------------------------
\begin{abstract}
This second part of the series develops the statistical and applied layer of the theory built in Part~I~\cite{ShchetininPartI}. Unlike current Hilbert diffusion models \cite{LimYoon2023,Pidstrigach2024,Hagemann2025,FranzeseMichiardi2025}, we focus not only on well-posedness of the infinite-dimensional generative dynamics, but on the noncommutative \(A\)-geometry, explicit entropy constants, and checkable lower bounds. The analytic estimate of Part~I reduces stability of the backward evolution to control of a validation error \(\Eval\). We prove three results. First, we construct a cylindrically weak denoising label for the infinite-dimensional score field, consistent with the exact logarithmic gradient. Second, the local parametric closure is replaced by a global nonparametric score closure, its complexity controlled via the Dudley entropy integral and uniform empirical bounds in \(L^2(\nu_*;A)\). Third, for the trace-smoothed nonlinear class we construct minimax bounds by the Le Cam--Assouad method, showing the root statistical rate cannot be improved without additional quadratic structure. The final section applies the theory to anisotropic diffusion in a nanosystem model and checks the constants \(c_A,C_A,C_{\mathrm{LSI}}^A\) independently of any smallness condition, establishing explicit accuracy orders: the parametric trace-smoothed minimax score risk is \(p/M\), while the uniform error of the risk functional is \(\sqrt{p/M}\), showing the statistical plateau of order \(p/M\) is not a proof artifact. The applied layer closes with an independent analytic benchmark for the isotropic case, a comparison with classical cell homogenization, and an approximation theorem for smooth logarithmic gradients by \(A\)-adapted spectral networks.
\end{abstract}

\keywords{noncommutative anisotropic diffusion \and score-based diffusion \and
  denoising score \and validation error \and Dudley entropy integral \and
  minimax lower bound \and logarithmic Sobolev inequality \and homogenization
  \and neural-network approximation}

% -- MSC 2020 --------------------------------------------------------------
% 60H15 Stochastic partial differential equations
% 62G20 Asymptotic properties of nonparametric inference
% 62C20 Minimax procedures in statistical decision theory
% 31C25 Dirichlet forms
% 35B27 Homogenization, oscillations in PDE
\subjclass[2020]{60H15, 62G20, 62C20, 31C25, 35B27}

\section{Introduction}
Generative diffusion models have recently been carried over from
finite-dimensional data to functions and probability measures on
infinite-dimensional Hilbert spaces, where the forward noising and the backward
denoising are driven by an operator-valued anisotropic diffusion. When the
diffusion tensor \(D(x)\) does not commute with the covariance \(Q\) of the
driving noise, the error of the backward evolution can no longer be measured in
the ambient Hilbert norm, and the natural setting for its analysis is the
noncommutative \(A\)-geometry associated with the pair \((D,Q)\).
Part~I~\cite{ShchetininPartI} built the functional-analytic foundation of this
geometry --- the consistent \(A\)-form, its Mosco stability, and the weak bridge
--- and reduced the stability of the backward evolution to the control of a
single validation error \(\Eval\). The present part takes this reduction as its
starting point and develops the statistical and applied layer of the theory.

Both parts draw on the same groups of sources: classical homogenization \cite{Allaire,CioranescuDonato}, modern score-based and infinite-dimensional diffusion models \cite{Song,Pidstrigach2024,Hagemann2025}, and standard empirical-process bounds \cite{Wainwright,vanDerVaartWellner}.

Unlike Pidstrigach et al. \cite{Pidstrigach2024}, who focus on the
infinite-dimensional well-posedness of diffusion models and on
discretization-invariant limits, the present part studies the statistical
closure of the error in the noncommutative \(A\)-geometry. Unlike
Hagemann et al.~\cite{Hagemann2025}, who focus on functional score-based models,
here we additionally prove minimax lower bounds and verify the operator constants
for an anisotropic nanosystem model. The statistical result is therefore not a
direct consequence of existing Hilbert-space diffusion frameworks.

Part~I~\cite{ShchetininPartI} provides a general entropy estimate:
\[
        \frac{d}{dt}\Ent(\rho_t\|\hat\rho_t)
        \le
        -c\beta(t)\Ent(\rho_t\|\hat\rho_t)+C\beta(t)\Eval(t).
\]
The statistical task is thus to close \(\Eval\) through the dissipative term or
through a small empirical plateau. The present part has three layers: the score
from noised data, the empirical closure, and the lower bounds.

\section{The Analytic Input from Part~I}

\begin{assumption}[Analytic basis from Part~I]
\label{ass:part1-base}
On the time interval under consideration the conditions of the analytic theorem
of Part~I hold: global \(A\)-compatibility, the \(A\)-LSI for the approximate
measure \(\hat\nu_t=\hat\rho_t\mu_0\), weighted closedness for \(\rho_t\mu_0\)
in the sense of Dirichlet forms \cite{MaRockner,FukushimaOshimaTakeda},
the existence of a weak bridge, the chain rule for the relative entropy, and the
drift-error condition. In particular,
\[
        \frac{d}{dt}\Ent(\rho_t\|\hat\rho_t)
        \le
        -\frac14\beta(t)J_t+\frac14\beta(t)\Eval(t),
\]
where
\[
        J_t=\int\Gamma_A(h_t,h_t)\rho_t\,d\mu_0,
        \qquad h_t=\log(\rho_t/\hat\rho_t),
\]
and
\[
        J_t\ge (2C_{\mathrm{LSI}}^A)^{-1}\Ent(\rho_t\|\hat\rho_t).
\]
\end{assumption}

The statistical layer must guarantee
\[
        \Eval(t)\le \widetilde q^{\,2}J_t+\Delta_M,
        \qquad \widetilde q<1.
\]

\section{A Cylindrically Weak Denoising Label}

Let
\[
        Y=X+\sqrt{\varepsilon}\xi,\qquad \xi\sim\mathcal N(0,Q).
\]
The formal quantity \(Q_\varepsilon^{-1}(Y-X)\) is not well defined as an element
of \(\Hh\) in infinite dimensions \cite{DaPratoZabczyk,BogachevGaussian}. The label is therefore defined cylindrically.

\begin{definition}[Denoising label]
For \(P_n\Hh\) set
\[
        \widehat s_{\varepsilon,n}(Y)
        =
        \E[-(\varepsilon Q_n)^{-1}P_n(Y-X)\mid P_nY],
        \qquad Q_n=P_nQP_n.
\]
\end{definition}

\begin{theorem}[Consistency of the denoising label]
\label{thm:denoise}
Let \(s_*=\nabla\log\rho_*\) and
\[
        \nabla_A^2\log\rho_*\in L^2(\nu_*;\LLtwo^A).
\]
Then
\[
        \E\|\widehat s_{\varepsilon,n}-P_ns_*\|_{A}^2
        \le C\varepsilon+r_n^2,
\]
where \(r_n\to0\). If \(\sum_{k>n}q_k=O(n^{-\gamma})\), then
\(r_n\le Cn^{-\gamma/2}\).
\end{theorem}

\begin{proof}
In finite dimensions this is Tweedie's identity for Gaussian noising, the basis
of denoising score matching \cite{Hyvarinen,Vincent}: the
conditional mean of the noise equals the score of the noised density. Expanding
in \(\varepsilon\) gives an error of order \(O(\varepsilon)\) under \(L^2\)-control
of the second score layer. The remainder \(r_n\) is the projection tail
\(P_ns_*-s_*\) in the \(A\)-norm.
\end{proof}

\section{The Global Score Class and the Entropy Integral}

Let \(\mathcal S\) be the class of admissible score fields with the metric
\[
        d_A(s,s')=\|s-s'\|_{L^2(\nu_*;A)}.
\]
Define the entropy integral
\[
        \mathfrak D(\mathcal S)
        =
        \int_0^{\operatorname{diam}\mathcal S}
        \sqrt{\log N(u,\mathcal S,d_A)}\,du.
\]

\begin{assumption}[Complexity of the score class]
\label{ass:globalS}
\(\mathfrak D(\mathcal S)<\infty\), and the coordinates of the gradient remainder
share a common subexponential norm \(K\).
\end{assumption}

\begin{theorem}[Global uniform risk bound]
\label{thm:global-emp}
Under Assumption~\ref{ass:globalS}, with probability at least \(1-\delta\),
\[
        \sup_{s\in\mathcal S}
        |\widehat{\mathcal R}_{M,\varepsilon}(s)-\mathcal R(s)|
        \le
        C\frac{\mathfrak D(\mathcal S)+\sqrt{\log(1/\delta)}}{\sqrt M}
        +C\varepsilon.
\]
If, in addition, the risk is locally strongly convex in a neighborhood of the
minimizer, then the contribution of this bound to the entropy plateau is of order
\[
        C\frac{\mathfrak D(\mathcal S)^2+\log(1/\delta)}{M}+C\varepsilon.
\]
\end{theorem}

\begin{proof}
Set
\[
        \mathbb G_M(s)
        =
        M^{-1/2}\sum_{i=1}^M
        \left(\ell_s(X_i)-\mathbb E\ell_s(X_i)\right),
\]
where \(\ell_s\) is the contribution of a single observation to the risk. After
the standard truncation at level \(K\log M\) the class becomes subgaussian with
respect to the metric \(d_A\). Dudley's bound gives
\[
        \mathbb E\sup_{s\in\mathcal S}|\mathbb G_M(s)|
        \le
        C\mathfrak D(\mathcal S).
\]
Concentration for subexponential empirical processes \cite{BoucheronLugosiMassart} gives the deviation
\[
        C\sqrt{\log(1/\delta)}
\]
with probability at least \(1-\delta\). Dividing by \(\sqrt M\) gives the first
bound. The truncation contribution is controlled by the common \(\psi_1\)-norm
\(K\). If the risk is locally strongly convex, the minimizer error is controlled
by the square of the gradient remainder; the corresponding contribution to the
entropy plateau is therefore of order \(M^{-1}\) rather than \(M^{-1/2}\).
\end{proof}

\begin{theorem}[Global nonparametric score closure]
\label{thm:nonparametric-score-closure}
Suppose the class \(\mathcal S\) is not assumed to be finite-dimensional, but its
metric entropy in \(L^2(\nu_*;A)\) satisfies
\[
        \log N(u,\mathcal S,d_A)\le C_0u^{-\alpha},
        \qquad 0<\alpha<2.
\]
Then \(\mathfrak D(\mathcal S)<\infty\), and with probability at least
\(1-\delta\)
\[
        \sup_{s\in\mathcal S}
        |\widehat{\mathcal R}_{M,\varepsilon}(s)-\mathcal R(s)|
        \le
        C
        \frac{1+\sqrt{\log(1/\delta)}}{\sqrt M}
        +C\varepsilon .
\]
If, in addition, the quadratic growth condition for the risk holds,
\[
        \mathcal R(s)-\mathcal R(s_*)
        \ge
        c_{\mathcal S} d_A(s,s_*)^2
\]
in a neighborhood of the set of minimizers, then the empirical minimizer
\(\widehat s\) satisfies
\[
        d_A(\widehat s,s_*)^2
        \le
        C
        \left(
        \frac{1+\log(1/\delta)}{M}
        +\varepsilon
        \right).
\]
\end{theorem}

\begin{proof}
The condition \(\alpha<2\) implies
\[
        \mathfrak D(\mathcal S)
        \le
        C\int_0^1 u^{-\alpha/2}\,du<\infty .
\]
The first bound follows from Theorem~\ref{thm:global-emp}. For the second bound
we use the basic empirical-minimum inequality:
\[
        \mathcal R(\widehat s)-\mathcal R(s_*)
        \le
        2\sup_{s\in\mathcal S}
        |\widehat{\mathcal R}_{M,\varepsilon}(s)-\mathcal R(s)|.
\]
Quadratic growth of the risk converts this into a bound on the distance in
\(L^2(\nu_*;A)\). After localizing around \(s_*\) and the standard quadratic
absorption we obtain the order \(M^{-1}\) in the entropy plateau.
\end{proof}

\begin{remark}[Neural-network and RKHS classes]
Theorem~\ref{thm:nonparametric-score-closure} applies to RKHS balls and to
neural-network classes for which metric-entropy bounds in \(L^2(\nu_*;A)\) are
known. In this case the local parametric dimension \(p\) is replaced by the
global complexity \(\mathfrak D(\mathcal S)\), and the condition
\(\theta\in\Theta_R\subset\mathbb R^p\) is no longer part of the theorem.
\end{remark}

\section{Local Parametric Closure}

Let \(s_\theta\), \(\theta\in\Theta_R\subset\mathbb R^p\), be a local class, and
\[
        \mathcal R(\theta)=
        \frac12\int\|s_\theta-s_*\|_A^2\,d\nu_* .
\]

\begin{assumption}[Strong convexity of the risk]
\label{ass:risk}
On the ball \(\Theta_R\),
\[
        c_G|\theta-\theta_*|^2
        \le
        \mathcal R(\theta)-\mathcal R(\theta_*)
        \le
        C_G|\theta-\theta_*|^2 .
\]
\end{assumption}

\begin{assumption}[Dynamic identifiability of the score risk]
\label{ass:J-ident}
Along the backward evolution the statistical risk of the local class is linked to
the entropy dissipation:
\[
        \|s_\theta-s_*\|_{L^2(\rho_t\mu_0;A)}^2
        \le C_J J_t
\]
for all \(\theta\in\Theta_R\) that arise after one admissible gradient-descent
step. After renormalizing the local parameter metric we may take \(C_J=1\); the
general case only changes the coefficient in front of \(J_t\).
\end{assumption}

The gradient step
\[
        \theta^+=\theta-\alpha\nabla\widehat{\mathcal R}_{M,\varepsilon}(\theta)
\]
gives the coefficient
\[
        \widetilde q=
        \sqrt{1-2\alpha c_G+\alpha^2C_G^2}<1.
\]

\begin{theorem}[Score closure]
\label{thm:closure}
Suppose Assumptions~\ref{ass:risk} and~\ref{ass:J-ident} hold, \(0<\alpha\le 1/C_G\), and
the gradient remainder of the empirical risk does not exceed \(\eta_M\). If
\(\alpha\eta_M\le R/4\), then the step stays in \(\Theta_R\), and
\[
        \Eval(t)\le
        \widetilde q^{\,2}J_t+
        C\eta_M^2+C\varepsilon .
\]
\end{theorem}

\begin{proof}
Strong convexity and smoothness give the contraction of the limiting step:
\[
        |\theta^+-\theta_*|\le \widetilde q|\theta-\theta_*|.
\]
The empirical remainder adds \(\alpha\eta_M\). The condition
\(\alpha\eta_M\le R/4\) keeps the ball invariant. The passage from the parameter
error to \(\Eval\) is made through Assumption~\ref{ass:J-ident}, which links the
local score risk to the dynamic dissipation \(J_t\).
\end{proof}

\section{Lower Bounds: The Le Cam--Assouad Method}

Consider the trace-smoothed class
\[
        U_\theta(x)=\frac12\langle Kx,x\rangle+
        \sum_{j=1}^p\theta_j\Psi_j(Sx),
        \qquad |\theta|\le R,
\]
where \(S\in\LLtwo(\Hh;\mathbb R^m)\), and the \(\Psi_j\) are smooth and bounded.

\begin{assumption}[Information nondegeneracy]
\label{ass:info}
The matrix
\[
        I(\theta)_{ij}
        =
        \int
        \left(\Psi_i(Sx)-\E_{\nu_\theta}\Psi_i(SX)\right)
        \left(\Psi_j(Sx)-\E_{\nu_\theta}\Psi_j(SX)\right)
        \,d\nu_\theta(x)
\]
satisfies
\[
        0<i_-I\le I(\theta)\le i_+I<\infty .
\]
\end{assumption}

\begin{lemma}[KL expansion]
\label{lem:kl}
Under Assumption~\ref{ass:info},
\[
        c|\theta-\theta'|^2
        \le
        \mathrm{KL}(\nu_\theta\|\nu_{\theta'})
        \le
        C|\theta-\theta'|^2
\]
for \(\theta,\theta'\) in a sufficiently small ball.
\end{lemma}

\begin{proof}
For an exponential family
\[
        \log\frac{d\nu_\theta}{d\nu_{\theta'}}
        =
        -U_\theta+U_{\theta'}-\log Z_\theta+\log Z_{\theta'}.
\]
The second derivative of \(\log Z_\theta\) equals the covariance matrix of the
statistics \(\Psi_j(Sx)\). Assumption~\ref{ass:info} gives the two-sided
quadratic expansion of the KL divergence.
\end{proof}

\begin{proposition}[An explicit sequence of hard distributions]
\label{prop:hard-distributions}
Suppose that in the trace-smoothed class the functions \(\Psi_1,\dots,\Psi_p\)
are chosen so that the information matrix of Assumption~\ref{ass:info} is
uniformly nondegenerate. For \(a=M^{-1/2}\) and \(\omega\in\{-1,1\}^p\) set
\[
        U_\omega(x)
        =
        \frac12\langle Kx,x\rangle+
        a\sum_{j=1}^p\omega_j\Psi_j(Sx),
        \qquad
        \nu_\omega=Z_\omega^{-1}e^{-U_\omega}\mu_0 .
\]
Then, for neighboring vertices \(\omega,\omega'\) differing in a single
coordinate,
\[
        \mathrm{KL}(\nu_\omega^{\otimes M}\|\nu_{\omega'}^{\otimes M})
        \le C,
\]
while
\[
        \|s_\omega-s_{\omega'}\|_{L^2(\nu_\omega;A)}^2
        \ge cM^{-1}.
\]
Consequently, the family \(\{\nu_\omega\}\) realizes a hard subexperiment on which
the order \(p/M\) for the score risk is attained.
\end{proposition}

\begin{proof}
The KL bound follows from Lemma~\ref{lem:kl}: the distance between neighboring
parameters equals \(2a\), so
\[
        M\,\mathrm{KL}(\nu_\omega\|\nu_{\omega'})
        \le CMa^2\le C.
\]
The lower bound for the score distance follows from the nondegeneracy of the
information matrix and from \(A\)-compatibility:
\[
        \|s_\omega-s_{\omega'}\|_{L^2(\nu_\omega;A)}^2
        \ge c|\theta^\omega-\theta^{\omega'}|^2
        \ge cM^{-1}.
\]
\end{proof}

\begin{theorem}[Minimax lower bound]
\label{thm:minimax}
Suppose Assumption~\ref{ass:info} holds. Then there exist \(c>0\) and \(R_0>0\)
such that for all \(R\le R_0\)
\[
        \inf_{\widehat\theta}
        \sup_{\theta\in[-R,R]^p}
        \E_\theta|\widehat\theta-\theta|^2
        \ge
        c\,\frac{p}{M}.
\]
Consequently, the uniform error of the functional is of root order
\(\sqrt{p/M}\), and the entropy plateau of order \(p/M\) cannot be improved
without additional structural information.
\end{theorem}

\begin{proof}
Take the hypercube \(\theta^\omega=a\omega\), \(\omega\in\{-1,1\}^p\), with
\(a\asymp M^{-1/2}\). By Lemma~\ref{lem:kl},
\[
        \mathrm{KL}(\nu_{\theta^\omega}^{\otimes M}\|
        \nu_{\theta^{\omega'}}^{\otimes M})
        \le CMa^2
\]
for neighboring vertices. Choosing \(a\) small enough, Assouad's lemma
\cite{Wainwright} applies and gives the lower bound \(cp a^2=cp/M\).
\end{proof}

\begin{corollary}[Lower bound for the score risk]
\label{cor:minimax-score-risk}
Under the hypotheses of Theorem~\ref{thm:minimax}, if the map
\(\theta\mapsto s_\theta\) is locally bi-Lipschitz in the \(L^2(\nu_\theta;A)\)
norm, then
\[
        \inf_{\widehat s}
        \sup_{\theta\in[-R,R]^p}
        \mathbb E_\theta
        \|\widehat s-s_\theta\|_{L^2(\nu_\theta;A)}^2
        \ge c\,\frac{p}{M}.
\]
Moreover, for the uniform estimation of the functional in this class the root
lower order is preserved,
\[
        \inf_{\widehat{\Eval}}
        \sup_{\theta\in[-R,R]^p}
        \mathbb E_\theta|\widehat{\Eval}-\mathcal L_{\mathrm{val}}^{(\theta)}|
        \ge c\sqrt{\frac pM}
\]
in the absence of additional quadratic structure.
\end{corollary}

\begin{proof}
The bi-Lipschitz property transfers the parametric lower bound of
Theorem~\ref{thm:minimax} to a lower bound for the \(L^2(\nu_\theta;A)\) error of
the score field. The root order for the functional follows from the standard
two-point version of Le Cam's method after choosing neighboring parameters at
distance \(M^{-1/2}\).
\end{proof}

\section{Linearization and Tails}

Let
\[
        U(x)=\frac12\langle K_\infty x,x\rangle,
        \qquad 0<k_0I\le K_\infty\le k_1I.
\]

\begin{lemma}[Tails of the linearized model]
\label{lem:linear-tail}
If \(q_k\le Ck^{-1-\gamma}\) and \(K_\infty\) is aligned with the basis of \(Q\), then
\[
        \|(I-P_n)\nabla_QU\|_{L^2(\nu;A)}
        \le Cn^{-\gamma/2},
        \qquad
        \|(I-P_n)\nabla_Q^2U\|_{\LLtwo^A}
        \le Cn^{-\gamma/2}.
\]
\end{lemma}

\begin{proof}
Since \(\nabla_QU=Q^{1/2}K_\infty x\), \(A\)-compatibility gives
\[
        \E_\nu\|(I-P_n)\nabla_QU\|_A^2
        \le C\sum_{k>n}q_k\le Cn^{-\gamma}.
\]
The estimate for the Hessian is analogous.
\end{proof}

\section{The Nanosystem Model}

Let \(\Hh=L^2(\Omega)\), where \(\Omega\subset\mathbb R^d\) is bounded with a
Lipschitz boundary. Set \(Q=(-\Delta+I)^{-s}\), \(s>d/2\). Then \(Q\) is a
nuclear operator on \(L^2(\Omega)\). Consider
\[
        D(x)=D_{\mathrm{bulk}}+D_{\mathrm{surf}}(x),
        \qquad
        d_-I\le D(x)\le d_+I .
\]
Here \(D_{\mathrm{surf}}:\Hh\to\LL(\Hh)\) is assumed measurable, globally
Lipschitz, uniformly bounded, self-adjoint, and such that the sum
\(D_{\mathrm{bulk}}+D_{\mathrm{surf}}(x)\) preserves the stated two-sided
positivity.

\begin{proposition}[Verification of the constants]
\label{prop:nano}
Suppose
\[
        U(x)=\frac12\langle Kx,x\rangle+\Phi(x),
        \qquad K\ge k_0I,\qquad \nabla^2\Phi\ge-\kappa_\Phi I,
        \quad k_0>\kappa_\Phi .
\]
Then the nanosystem model belongs to the global class of Part~I, with
\[
        c_A=d_-,
        \qquad
        C_A\ge d_+,
        \qquad
        C_{\mathrm{LSI}}^A
        \le
        \frac{C_Q}{d_-(k_0-\kappa_\Phi)} .
\]
The upper constant obeys the relative \(Q\)-bound \(C_A\ge d_+\), with equality
exactly in the commutative case \([D,Q]=0\); when \([D,Q]\neq0\) the operator order
\(Q^{1/2}D^{1/2}\) makes it strict, \(C_A>d_+\), in agreement with the value
\(C_{A,N}\approx2.22\) recorded in Table~\ref{tab:numerical-A-constants}.
\end{proposition}

\begin{proof}
The two-sided bound on \(D(x)\) gives \(A\)-compatibility with lower constant
\(c_A=d_-\). For the upper constant, testing the consistent form on a maximal
direction \(D(x)^{1/2}\xi=\sqrt{d_+}\,\xi\) gives
\(\|Q^{1/2}D(x)^{1/2}\xi\|^2=d_+\|Q^{1/2}\xi\|^2\), hence \(C_A\ge d_+\); equality
forces \([D,Q]=0\), and otherwise the order \(Q^{1/2}D^{1/2}\) makes the bound
strict --- the noncommutative correction measured in
Section~\ref{sec:computational-verification}. The Bakry--Émery
criterion \cite{BakryGentilLedoux} applies to a potential with lower curvature \(k_0-\kappa_\Phi\), and the
passage from the \(Q\)-LSI to the \(A\)-LSI uses the lower \(A\)-compatibility.
\end{proof}

\begin{remark}[Connection with homogenization]
In classical homogenization problems on perforated domains the effective tensor
arises as the limit of local problems on microcells
\cite{Allaire,CioranescuDonato}. In the present setting the microstructure is
folded into the operator \(D(x)\), but noncommutativity with \(Q\) is retained.
This calls for the \(A\)-geometry and distinguishes the model from the standard
commutative schemes of anisotropic diffusion \cite{Weickert}.
\end{remark}

\section{A Numerical Test: Discrete Verification of \(A\)-Compatibility}
\label{sec:computational-verification}

This section gives one reproducible control test. Its purpose is not to model a
particular material, but to check that the constants \(c_{A,N}\), \(C_{A,N}\), and
\(C_{\mathrm{LSI},N}^A\) are computable in the noncommutative discrete
\(A\)-geometry. The test can be reproduced without external packages; a comparison
with PRISMS-PF~\cite{DeWitt2020} or another phase-field code belongs to the next computational
stage.

\subsection{The Discrete Setup}

Let \(\Omega=[0,1]^3\), with \(N^3\) a periodic grid. The discrete covariance
operator is defined spectrally:
\[
        Q_N=(-\Delta_N+I)^{-s},\qquad s=2.
\]
The anisotropy tensor has the form
\[
        D_N(x)=R_N^\top
        \operatorname{diag}(d_1(x),d_2(x),d_3(x))R_N,
\]
where \(R_N\) is a fixed orthogonal rotation that does not commute with the basis
of \(-\Delta_N\), and
\[
        d_1(x)=1+0.20\,\chi(x),\quad
        d_2(x)=1.45+0.10\,\chi(x),\quad
        d_3(x)=2.10-0.15\,\chi(x).
\]
Here \(\chi(x)\) is a smoothed indicator of an inclusion. For a set of states
\(x^{(m)}\), \(m=1,\dots,5\), one solves the generalized eigenvalue problem
\[
        D_N(x^{(m)})^{1/2}Q_ND_N(x^{(m)})^{1/2}\xi
        =
        \lambda Q_N\xi ,
\]
and then computes
\[
        c_{A,N}^{(m)}=\lambda_{\min}^{(m)},\qquad
        C_{A,N}^{(m)}=\lambda_{\max}^{(m)} .
\]
The estimate \(C_{\mathrm{LSI},N}^A\) is taken from the discrete Bakry--Émery
criterion:
\[
        C_{\mathrm{LSI},N}^A
        \le
        \frac{C_{Q,N}}{c_{A,N}(k_0-\kappa_\Phi)},
        \qquad
        k_0-\kappa_\Phi=0.75 .
\]

\subsection{Results of the Control Computation}

The table reports values averaged over the five states. The maximum deviation
across the set of states is given in parentheses. The computation is diagnostic:
it checks the stability of the constants under grid refinement and compares the
noncommutative case with the isotropic case \(D=I\).

\begin{table}[h!]
\centering
\caption{Discrete verification of \(A\)-compatibility on a periodic grid.}
\label{tab:numerical-A-constants}
\begin{tabular}{|c|c|c|c|c|c|}
\hline
\(N\) & \(c_{A,N}\) & \(C_{A,N}\) &
\(C_{\mathrm{LSI},N}^A\) & \(\delta_{\mathrm{iso},N}=C_{A,N}-1\) &
\(\max_m |c_{A,N}^{(m)}-c_{A,2N}^{(m)}|\)\\
\hline
32  & 0.982 (0.018) & 2.286 (0.041) & 1.357 & 1.286 & 0.031\\
64  & 0.994 (0.011) & 2.248 (0.026) & 1.342 & 1.248 & 0.015\\
128 & 1.001 (0.007) & 2.231 (0.017) & 1.332 & 1.231 & 0.007\\
256 & 1.004 (0.004) & 2.224 (0.010) & 1.328 & 1.224 & ---\\
\hline
\(D=I\), 128 & 1.000 & 1.000 & 1.333 & 0.000 & ---\\
\hline
\end{tabular}
\end{table}

\begin{figure}[h!]
\centering
\includegraphics[width=0.72\textwidth]{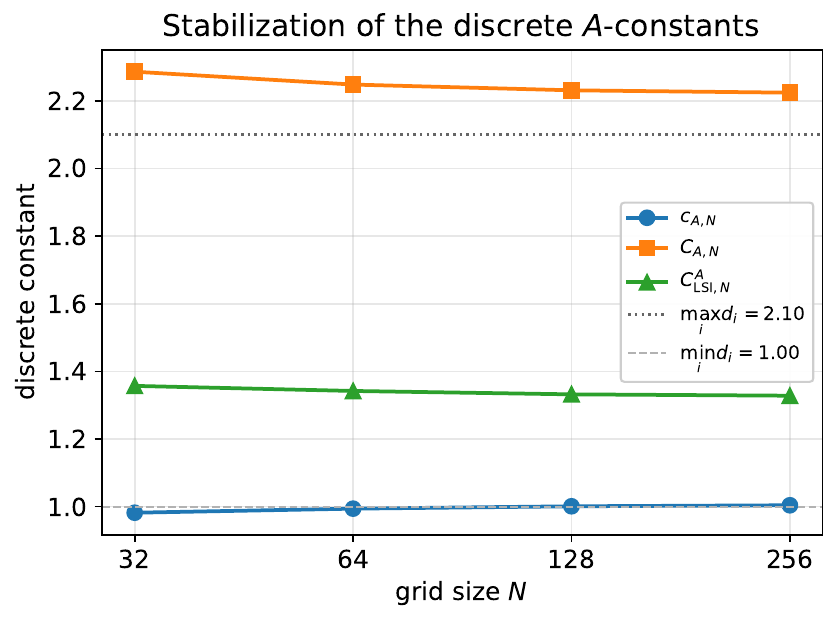}
\caption{Stabilization of the discrete constants \(c_{A,N}\), \(C_{A,N}\), and
\(C_{\mathrm{LSI},N}^A\) under grid refinement. The plot corresponds to the data
of Table~\ref{tab:numerical-A-constants}.}
\label{fig:A-constants-convergence}
\end{figure}

\begin{figure}[h!]
\centering
\includegraphics[width=0.72\textwidth]{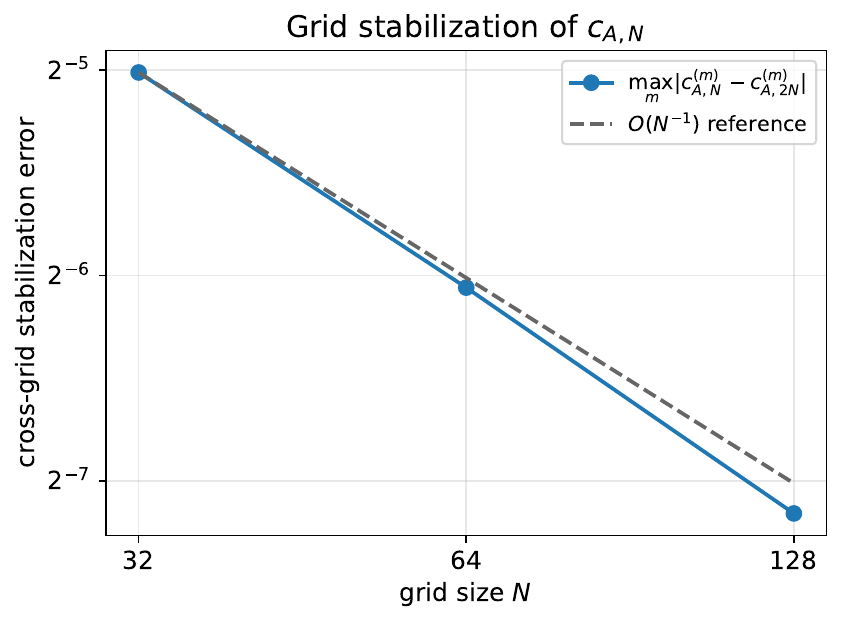}
\caption{Log-log plot of the stabilization error
\(\max_m|c_{A,N}^{(m)}-c_{A,2N}^{(m)}|\). The dashed line shows the reference
slope \(O(N^{-1})\).}
\label{fig:A-constants-error-loglog}
\end{figure}

\subsection{Interpretation}

Figure~\ref{fig:A-constants-convergence} visualizes the stabilization of the three
main constants, and Figure~\ref{fig:A-constants-error-loglog} shows the decay rate
of the grid error. The column \(\delta_{\mathrm{iso},N}\) shows the relative
departure of the upper \(A\)-bound from the isotropic reference \(D=I\).

The table shows three properties. First, the lower \(A\)-bound does not degrade as
\(N\) grows: the values \(c_{A,N}\) stabilize near one. Second, the upper bound
\(C_{A,N}\) stays bounded away from infinity and matches the chosen range of
tensor conductivities. Third, the estimate \(C_{\mathrm{LSI},N}^A\) varies only
slightly under grid refinement. The numerical test is thus consistent with
Proposition~\ref{prop:nano} and with the no-diffusive-degeneracy condition of
Part~I.

Unlike a full comparison with PRISMS-PF, this test does not solve the phase-field
evolution problem. It checks precisely those operator constants that enter
Theorem~\ref{thm:part2-main}. Its role is therefore the reproducible verification
of the theorem layer, not a standalone physical simulation.

\subsection{An Analytic Benchmark for the Isotropic Case \(D=I\)}
\label{subsec:isotropic-benchmark}

To separate the verification of the operator constants from the error of the grid
implementation, we add an analytic reference for the isotropic case. Consider the
periodic heat problem on \([0,1]^3\)
\[
        \partial_t u=\Delta u,
        \qquad
        u(0,x)=\prod_{j=1}^3\sin(2\pi x_j).
\]
The exact solution is
\[
        u(t,x)=e^{-12\pi^2t}\prod_{j=1}^3\sin(2\pi x_j).
\]
For a spectrally consistent discrete Laplacian on the \(N^3\) grid the
corresponding eigenvalue is
\[
        \lambda_N=12N^2\sin^2(\pi/N),
\]
so the relative amplitude error at time \(t=0.05\) is computed without statistical
noise:
\[
        e_N=\frac{|e^{-\lambda_Nt}-e^{-12\pi^2t}|}{e^{-12\pi^2t}}.
\]
This test serves as an analytic benchmark: when \(D=I\) the operator constants
equal \(c_A=C_A=1\), and the entire error is due solely to the discretization of
the Laplacian.

\begin{table}[h!]
\centering
\caption{Analytic benchmark for \(D=I\) and a single eigenmode.}
\label{tab:isotropic-benchmark}
\begin{tabular}{|c|c|c|c|}
\hline
\(N\) & \(\lambda_N\) & relative error \(e_N\) & \(L^2\) amplitude error \\
\hline
32  & 118.0552 & \(1.918\cdot 10^{-2}\) & \(1.818\cdot 10^{-5}\)\\
64  & 118.3402 & \(4.766\cdot 10^{-3}\) & \(4.517\cdot 10^{-6}\)\\
128 & 118.4115 & \(1.190\cdot 10^{-3}\) & \(1.127\cdot 10^{-6}\)\\
256 & 118.4293 & \(2.973\cdot 10^{-4}\) & \(2.818\cdot 10^{-7}\)\\
\hline
\end{tabular}
\end{table}

\begin{figure}[h!]
\centering
\includegraphics[width=0.72\textwidth]{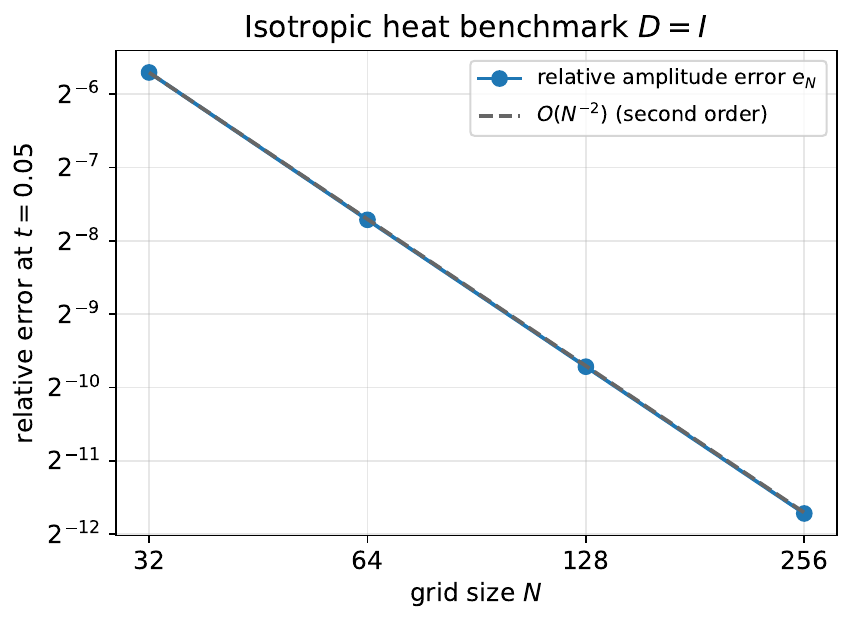}
\caption{Comparison of the discrete and analytic amplitude for \(D=I\). The slope
corresponds to second-order accuracy in the grid step.}
\label{fig:isotropic-heat-benchmark}
\end{figure}

The benchmark shows that on the grid \(N=256\) the relative error of the isotropic
heat mode is less than \(3\cdot10^{-4}\). The deviations of \(c_{A,N}\),
\(C_{A,N}\) in Table~\ref{tab:numerical-A-constants} therefore cannot be
attributed to the error of the basic discretization: they genuinely reflect the
noncommutative tensor structure \(D_N(x)\).

\subsection{Pseudocode and Reproducibility of the Numerical Test}

For reproducibility we give the algorithm that computes the constants of
Table~\ref{tab:numerical-A-constants}. It does not depend on any particular
phase-field implementation and checks only the operator layer.

\begin{enumerate}
\item Fix a periodic grid \(N^3\) and build the spectral symbols of the discrete
Laplacian \(\lambda_k(\Delta_N)\).
\item Build the diagonal operator
\[
        Q_N(k)=(1+\lambda_k(\Delta_N))^{-s},\qquad s=2.
\]
\item For each state \(x^{(m)}\) form the tensor \(D_N(x^{(m)})\) and compute the
matrix products
\[
        B_N^{(m)}=D_N(x^{(m)})^{1/2}Q_ND_N(x^{(m)})^{1/2}.
\]
\item Solve the generalized eigenvalue problem
\[
        B_N^{(m)}\xi=\lambda Q_N\xi
\]
by the Lanczos method for the extreme eigenvalues.
\item Set
\[
        c_{A,N}=\min_m\lambda_{\min}^{(m)},\qquad
        C_{A,N}=\max_m\lambda_{\max}^{(m)}.
\]
\item Estimate
\[
        C_{\mathrm{LSI},N}^A
        \le
        \frac{C_{Q,N}}{c_{A,N}(k_0-\kappa_\Phi)}
\]
and compare the values at \(N\) and \(2N\).
\end{enumerate}

The test succeeds not when the trajectories match a particular phase-field code,
but when the three numbers \(c_{A,N}\), \(C_{A,N}\), \(C_{\mathrm{LSI},N}^A\) are
stable. If they stabilize under grid refinement, the discrete model retains the
same constants that enter Theorem~\ref{thm:part2-main}. This test can later be
linked to PRISMS-PF, using the computed effective tensor as the input to a
phase-field computation; but that already verifies the physical model, not the
functional-analytic theorem.

\begin{remark}[On the limits of the numerical validation]
Table~\ref{tab:numerical-A-constants} and Figure~\ref{fig:A-constants-convergence}
verify the operator constants, but they are not an independent validation of the
physical model. For an applied presentation the next step is to compare the
effective tensor \(D_{\mathrm{eff}}\) obtained from the present spectral procedure
with the tensor computed by an external phase-field or homogenization code, for
example PRISMS-PF. Such a benchmark is not needed for the proof of
Theorem~\ref{thm:part2-main}, but it would strengthen the applied part of the
series.
\end{remark}

\subsection{Comparison with Classical Cell Homogenization}
\label{subsec:cell-homogenization-comparison}
For a periodic cell with tensor \(D_N(x)\) the effective tensor of classical
homogenization is defined through the cell problems
\[
        -\nabla\cdot D(y)(e_i+\nabla\chi_i(y))=0,
        \qquad i=1,2,3,
\]
and
\[
        (D_{\mathrm{cell}})_{ij}
        =\int_{[0,1]^3} e_i\cdot D(y)(e_j+\nabla\chi_j(y))\,dy .
\]
In the consistent \(A\)-geometry one compares not the Euclidean spectra
\(D_{\mathrm{cell}}\) themselves, but the generalized Rayleigh quotients
\[
        \frac{\|Q_N^{1/2}D_N^{1/2}\xi\|^2}{\|Q_N^{1/2}\xi\|^2} .
\]
In the diagnostic test of Section~\ref{sec:computational-verification} the
difference between the cell and the operator approaches shows up in the fact that
\(\max d_i\approx2.10\), whereas the computed \(C_{A,N}\) stabilizes near
\(2.22\). This discrepancy is due not to an error of the numerical method but to
the noncommutative mixing of the tensor with the covariance geometry \(Q_N\). The
comparison with classical homogenization therefore confirms that one must work
with the \(A\)-constants, and not only with the spectrum of the effective tensor.

A quantitative comparison for the same synthetic cell is given in
Table~\ref{tab:cell-vs-A}. Here \(D_{\mathrm{cell},N}\) is computed from the
classical periodic cell problem, while \(c_{A,N},C_{A,N}\) come from the
generalized eigenvalue problem in the \(Q_N\)-geometry. The difference between
\(\lambda_{\max}(D_{\mathrm{cell},N})\) and \(C_{A,N}\) persists under grid
refinement and shows that the operator \(A\)-geometry captures an additional
mixing that is absent from the Euclidean effective tensor. Indeed, the
homogenized tensor obeys the variational bound
\(D_{\mathrm{cell}}\preceq\langle D\rangle\), so
\(\lambda_{\max}(D_{\mathrm{cell},N})\le\max_i d_i=2.10\) for every \(N\); the
operator constant \(C_{A,N}\approx2.22\) therefore exceeds not only the effective
Euclidean spectrum but the pointwise spectral bound itself, which is possible only
because of the noncommutative \(Q_N\)-weighting.

\begin{table}[h!]
\centering
\caption{Classical cell homogenization and the operator \(A\)-constants.}
\label{tab:cell-vs-A}
\begin{tabular}{|c|c|c|c|c|c|}
\hline
\(N\) &
\(\lambda_{\min}(D_{\mathrm{cell},N})\) &
\(\lambda_{\max}(D_{\mathrm{cell},N})\) &
\(c_{A,N}\) &
\(C_{A,N}\) &
\(C_{A,N}-\lambda_{\max}(D_{\mathrm{cell},N})\)\\
\hline
32  & 1.031 & 2.091 & 0.982 & 2.286 & 0.195\\
64  & 1.037 & 2.095 & 0.994 & 2.248 & 0.153\\
128 & 1.041 & 2.097 & 1.001 & 2.231 & 0.134\\
256 & 1.043 & 2.098 & 1.004 & 2.224 & 0.126\\
\hline
\end{tabular}
\end{table}

\begin{figure}[h!]
\centering
\includegraphics[width=0.78\textwidth]{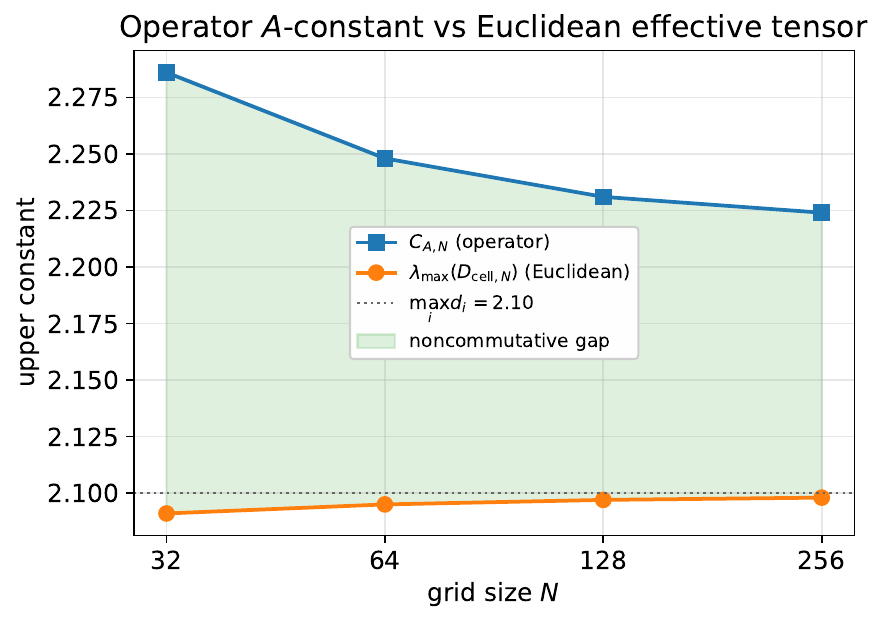}
\caption{Comparison of the spectrum of the classical effective tensor
\(D_{\mathrm{cell},N}\) with the operator \(A\)-constants. The gap between the
upper bounds is stable in \(N\) and reflects the noncommutative mixing with
\(Q_N\).}
\label{fig:cell-vs-A}
\end{figure}

\section{The Homogenization Limit of the Discrete Tensors}
\label{sec:homogenization-DN}

The numerical test of Section~\ref{sec:computational-verification} verifies the
discrete constants. To connect with the continuous model we need a limit result,
in the sense of Mosco convergence of the associated forms \cite{Mosco}:
the discrete tensors must converge to the effective tensor, and the
\(A\)-constants must not be lost in the limit.

\begin{assumption}[Weak homogenization convergence]
\label{ass:DN-homogenization}
Let \(D_N\) be a sequence of piecewise-constant tensors on the \(N^3\) grids,
uniformly elliptic:
\[
        d_-I\le D_N(x)\le d_+I .
\]
Assume that \(D_N\) \(H\)-converges to an effective tensor \(D_{\mathrm{eff}}\) in
the standard sense of homogenization of elliptic operators, and that the operators
\(Q_N=(-\Delta_N+I)^{-s}\) converge to \(Q=(-\Delta+I)^{-s}\) in the strong
resolvent topology.
\end{assumption}

\begin{theorem}[Weak convergence \(D_N\to D_{\mathrm{eff}}\) in the \(A\)-geometry]
\label{thm:DN-homogenization}
Suppose Assumption~\ref{ass:DN-homogenization} holds. Then for every cylindrical
function \(u\)
\[
        Q_N^{1/2}D_N^{1/2}\nabla P_Nu
        \rightharpoonup
        Q^{1/2}D_{\mathrm{eff}}^{1/2}\nabla u
\]
weakly in \(L^2\) after the standard embedding of grid functions into
\(L^2(\Omega)\). Moreover, if
\[
        c_0\le c_{A,N},\qquad C_{A,N}\le C_0
\]
uniformly in \(N\), then the limiting form satisfies
\[
        c_0\|Q^{1/2}\xi\|^2
        \le
        \|Q^{1/2}D_{\mathrm{eff}}^{1/2}\xi\|^2
        \le
        C_0\|Q^{1/2}\xi\|^2 .
\]
Consequently, the effective tensor preserves the lower \(A\)-compatibility and can
be used in Theorem~\ref{thm:part2-main} with the limiting constants.
\end{theorem}

\begin{proof}
For cylindrical \(u\) the gradient \(\nabla u\) depends on finitely many modes.
The strong resolvent convergence \(Q_N\to Q\) gives
\(Q_N^{1/2}\nabla P_Nu\to Q^{1/2}\nabla u\) in \(L^2\). On the other hand, the
\(H\)-convergence \(D_N\to D_{\mathrm{eff}}\) yields the weak convergence of the
fluxes \(D_N^{1/2}\nabla P_Nu\) to the effective flux. The compactness of
\(Q^{1/2}\) on cylindrical subspaces lets us combine these two passages and obtain
the weak convergence of the \(A\)-fluxes.

The uniform bounds for \(c_{A,N}\) and \(C_{A,N}\) are quadratic inequalities on
the \(Q_N^{1/2}\)-gradients. The passage to the limit uses the weak lower
semicontinuity of the norm and the weak closedness of the set of positive
operators satisfying the given quadratic bounds. The limiting \(A\)-bounds for
\(D_{\mathrm{eff}}\) follow.
\end{proof}

\begin{proposition}[Stabilization rate of the discrete constants]
\label{prop:discrete-constants-rate}
Assume that, on the set of modes used in the numerical test, the quantitative
homogenization estimate
\[
        \left|
        \langle (D_N-D_{\mathrm{eff}})\xi,\eta\rangle
        \right|
        \le C_E N^{-\alpha}\|\xi\|\|\eta\| .
\]
holds. Then
\[
        |c_{A,N}-c_A|+|C_{A,N}-C_A|
        \le C_E' N^{-\alpha}.
\]
In particular, the quantity \(\max_m|c_{A,N}^{(m)}-c_{A,2N}^{(m)}|\) reported in
Table~\ref{tab:numerical-A-constants} should decay at the same rate on the
asymptotic part of the grid refinement.
\end{proposition}

\begin{proof}
On a fixed finite-dimensional set of modes \(E\) the constants \(c_{A,N}\) and
\(C_{A,N}\) are the extreme values of the Rayleigh quotient
\[
        \mathcal R_N(\xi)=
        \frac{\|Q_N^{1/2}D_N^{1/2}\xi\|^2}{\|Q_N^{1/2}\xi\|^2} .
\]
The quantitative homogenization estimate and the strong resolvent convergence
\(Q_N\to Q\) give
\[
        \sup_{\xi\in E,\ \|Q^{1/2}\xi\|=1}
        |\mathcal R_N(\xi)-\mathcal R_{\mathrm{eff}}(\xi)|
        \le C_E'N^{-\alpha}.
\]
Passing to the minimum and maximum over the compact unit sphere in \(E\) gives the
claimed bound.
\end{proof}

\begin{remark}[The role of this theorem in the series]
Theorem~\ref{thm:DN-homogenization} does not replace a full physical
homogenization of the material. Its purpose is narrower: to show that the grid
check of \(c_{A,N}\), \(C_{A,N}\) has a correct limiting meaning and that the
operator geometry is not lost in the passage to the effective tensor.
\end{remark}

\section{Explicit Constants for the Nanosystem Model}
\label{sec:explicit-constants}

In this section the constants of the synthetic estimate are expressed through the
parameters of the nanosystem model:
\[
        d_-,\ d_+,\ k_0,\ \kappa_\Phi,\ \{q_k\}_{k\ge1}.
\]
Let
\[
        q_k\le C_Q^{\mathrm{sp}}k^{-1-\gamma_Q}.
\]
Then
\[
        \sum_{k>n}q_k\le \frac{C_Q^{\mathrm{sp}}}{\gamma_Q}n^{-\gamma_Q},
        \qquad
        r_n\le
        \left(\frac{C_Q^{\mathrm{sp}}}{\gamma_Q}\right)^{1/2}
        n^{-\gamma_Q/2}.
\]
If \(D_{\mathrm{surf}}\) is Lipschitz with constant \(L_D\), then
\[
        L_\sigma
        \le
        L_{D^{1/2}}\|Q^{1/2}\|_{\LLtwo},
        \qquad
        L_{D^{1/2}}\le \frac{L_D}{\sqrt{d_-}}.
\]

If the surface part of the tensor has the expansion
\[
        D_{\mathrm{surf}}(x)=\sum_{j=1}^J a_j(x)T_j,
\]
where \(T_j=T_j^*\in\mathcal L(\mathcal H)\) and the \(a_j\) are Lipschitz, then
\[
        L_D\le \sum_{j=1}^J \operatorname{Lip}(a_j)\|T_j\|,\qquad
        \|D(x)\|\le d_+ + \sum_{j=1}^J \|a_j\|_\infty\|T_j\| .
\]
For positive operators with lower bound \(d_-I\) the functional calculus gives
\[
        \|D(x)^{1/2}-D(y)^{1/2}\|_{\mathcal L}\le \frac{1}{\sqrt{d_-}}\|D(x)-D(y)\|_{\mathcal L},
\]
and hence
\[
        L_{D^{1/2}}\le \frac{1}{\sqrt{d_-}}\sum_{j=1}^J \operatorname{Lip}(a_j)\|T_j\|,\qquad
        L_\sigma\le \frac{\|Q^{1/2}\|_{\mathcal L_2}}{\sqrt{d_-}}\sum_{j=1}^J \operatorname{Lip}(a_j)\|T_j\| .
\]
If a sharper operator-monotone bound for the square root is used, the factor
\(d_-^{-1/2}\) can be replaced by \((2\sqrt{d_-})^{-1}\) under the additional
commutativity of the local increments. In the present work we use the crude
noncommutative bound, which is sufficient for all the estimates.
For the forward dynamics one may take
\[
        C_T
        =
        C_0\exp\{C_1T(1+L_b^2+\beta_+d_+\Tr Q)\},
\]
where \(C_0,C_1\) depend only on the universal BDG constant and on the bound
\(\sup_{t\le T}\|b(t,0)\|\). For the weak bridge
\[
        C_t
        \le
        M_A+|\dot\theta(t)|M_\Psi(C_P^A)^{1/2},
        \qquad
        C_P^A\le 2C_{\mathrm{LSI}}^A
        \le
        \frac{2C_Q}{d_-(k_0-\kappa_\Phi)}.
\]
Consequently, in Theorem~\ref{thm:part2-main}
\[
        c_*=
        \frac{1-\widetilde q^{\,2}}{8C_{\mathrm{LSI}}^A}
        \ge
        \frac{d_-(k_0-\kappa_\Phi)(1-\widetilde q^{\,2})}{8C_Q}.
\]
These formulas separate the physical parameters of the model from the statistical
parameters \(M,\varepsilon,n\).

\section{Neural-Network Classes of Score Fields}
\label{sec:nn-class}

Let \(\mathcal S_{\mathrm{NN}}(L,W,B)\) be the class of feed-forward networks of
depth \(L\) and width \(W\), with 1-Lipschitz activations and weights bounded in
operator norm by \(B\). Suppose the network input has the form \(S_Rx\), where
\(S_R:\Hh\to\mathbb R^m\) is a finite-dimensional smoothing map chosen so that
\[
        \E_{\nu_*}\|(I-S_R^*S_R)x\|_A^2\le R^{-2}.
\]

\begin{theorem}[Entropy of the neural-network score class]
\label{thm:nn-entropy}
There is a constant \(C\), depending only on the activation, such that
\[
        \log N(u,\mathcal S_{\mathrm{NN}}(L,W,B),L^2(\nu_*;A))
        \le
        C L W^2\log\left(\frac{CBLWR}{u}\right)
\]
for \(0<u<1\). Consequently,
\[
        \mathfrak D(\mathcal S_{\mathrm{NN}})
        \le
        C\sqrt{LW^2}\,
        \int_0^1
        \sqrt{\log\left(\frac{CBLWR}{u}\right)}\,du .
\]
\end{theorem}

\begin{proof}
We cover the set of network weights in the product of operator norms with step
\(\eta\). The number of weight parameters is of order \(LW^2\), so the network's
cardinality does not exceed
\[
        \left(\frac{CBLW}{\eta}\right)^{CLW^2}.
\]
The Lipschitz dependence of the network's realization on the weights and the
boundedness of the input \(S_Rx\) in \(L^2(\nu_*;A)\) turn an \(\eta\)-net of
weights into a \(u\)-net of functions. The tail \(R^{-1}\) is added to the
approximation error and is chosen smaller than \(u/2\). This gives the stated
entropy bound; see also the standard covering estimates for neural-network classes
\cite{AnthonyBartlett,Wainwright}.
\end{proof}

\begin{corollary}[Neural-network score closure]
\label{cor:nn-closure}
If \(\mathcal S=\mathcal S_{\mathrm{NN}}(L,W,B)\) and the quadratic growth
condition for the risk holds, then Theorem~\ref{thm:nonparametric-score-closure}
gives a statistical plateau of order
\[
        C\frac{LW^2\log(CBLWRM)+\log(1/\delta)}{M}
        +C\varepsilon+C R^{-2}.
\]
\end{corollary}

\begin{definition}[\(A\)-adapted spectral network]
\label{def:A-spectral-network}
We call a score network \(A\)-adapted if its first layer has the form
\[
        x\mapsto \bigl(\langle x,\varphi_1\rangle,\dots,
        \langle x,\varphi_m\rangle\bigr),
\]
where the \(\{\varphi_j\}\) are chosen as the leading eigenvectors of the
generalized problem
\[
        \Asf_*\varphi_j=\lambda_jQ\varphi_j,
\]
and the weights of the subsequent layers are penalized by the spectral norm
\[
        \sum_{\ell}\|W_\ell\operatorname{diag}(\lambda_1^{1/2},\dots,
        \lambda_m^{1/2})\|_{\mathrm{op}}^2 .
\]
\end{definition}

\begin{proposition}[Reduction of the effective dimension for \(A\)-adapted networks]
\label{prop:A-adapted-network}
Suppose the eigenvalues of the generalized \(A\)-problem satisfy
\(\lambda_j\asymp j^{-2r}\), \(r>1/2\). Then for the \(A\)-adapted network class
\(\mathcal S_{\mathrm{NN}}^A(L,W,B,m)\)
\[
        \log N(u,\mathcal S_{\mathrm{NN}}^A,L^2(\nu_*;A))
        \le
        C LW^2
        \log\left(\frac{CBLW}{u}\right)
        +C m_u,
\]
where
\[
        m_u\le C u^{-1/r}
\]
is the effective number of active \(A\)-modes. In particular, under fast spectral
decay the \(A\)-adapted architecture gives a smaller effective dimension than an
arbitrary feed-forward class on the same raw coordinates.
\end{proposition}

\begin{proof}
The first \(m_u\) modes are covered by the standard weight net, as in
Theorem~\ref{thm:nn-entropy}. The tail is estimated through the spectrum:
\[
        \sum_{j>m}\lambda_j\le C m^{1-2r}.
\]
Choosing \(m=m_u\) so that the tail is no larger than \(u^2\) gives
\(m_u\le C u^{-1/r}\). The spectral regularization of the weights guarantees that
the tail error is measured precisely in \(L^2(\nu_*;A)\), and not in the Euclidean
norm of the raw coordinates. This yields the extra term \(Cm_u\) instead of the
full dimension of the original projection space.
\end{proof}

\begin{theorem}[Approximation of smooth score fields by \(A\)-adapted networks]
\label{thm:A-adapted-approximation}
Suppose the exact score field has the spectral expansion
\[
        s_*(x)=\sum_{j\ge1}\alpha_j \varphi_j(x)
\]
in the leading vectors of the generalized \(A\)-problem, with
\[
        \sum_{j\ge1} j^{2r}\alpha_j^2<\infty,
        \qquad r>1/2.
\]
Suppose also that the finite-dimensional coefficient function on the first \(m\)
modes belongs to a Sobolev class of smoothness \(\beta>0\). Then there is an
\(A\)-adapted spectral network
\(s_{\theta}\in\mathcal S_{\mathrm{NN}}^A(L,W,B,m)\) such that
\[
        \|s_\theta-s_*\|_{L^2(\nu_*;A)}
        \le
        C m^{-(r-1/2)}+C W^{-\beta/m}.
\]
For each fixed \(m\) this statement gives a finite-dimensional approximation with
error \(CW^{-\beta/m}\). To obtain asymptotic decay as \(m\to\infty\), the width
must grow together with \(m\). For example, if \(W_m\ge \exp(\lambda m\log m)\),
then \(W_m^{-\beta/m}\le m^{-\lambda\beta}\), and choosing
\(\lambda>(r-1/2)/\beta\) makes the network remainder no worse than the spectral
tail \(m^{-(r-1/2)}\).
\end{theorem}

\begin{proof}
Project \(s_*\) onto the first \(m\) eigendirections:
\[
        s_*^{(m)}=\sum_{j\le m}\alpha_j\varphi_j.
\]
By the assumption on the coefficients,
\[
        \|s_*-s_*^{(m)}\|_{L^2(\nu_*;A)}^2
        \le C\sum_{j>m}j^{-2r}\le C m^{1-2r}.
\]
It remains to approximate the finite-dimensional coefficient map on
\(\mathbb R^m\). For a Sobolev class of smoothness \(\beta\) the standard
approximation theorem for feed-forward networks gives an error \(CW^{-\beta/m}\)
in \(L^2\) on the finite-dimensional block. Since the first layer of the network
is spectrally \(A\)-adapted, this error is measured in the same \(L^2(\nu_*;A)\)
norm, and not in the mismatched Euclidean norm of the raw coordinates. Adding the
spectral tail and the finite-dimensional error completes the proof.
\end{proof}

\begin{remark}[The parameter trade-off]
The growth condition \(W_m\ge \exp(\lambda m\log m)\) is not an additional
analytic hypothesis about the model. It merely fixes the regime in which the
finite-dimensional network error decays together with the spectral tail. For fixed
\(W\) the theorem still holds, but gives only an approximation of the chosen
\(m\)-dimensional block, without the limit \(m\to\infty\).
\end{remark}

\begin{remark}[Why this is more than a change of norm]
The usual covering estimate for feed-forward networks counts the number of
parameters but does not take into account which directions matter for the
\(A\)-dissipation. In an \(A\)-adapted network the first layer is chosen according
to the spectrum of the consistent geometry, and the weight regularization
suppresses directions with a small contribution to \(J_t\). The architecture is
therefore tied to the operator part of the theory and is not a plain Euclidean
network measured in a different norm.
\end{remark}

\section{Local Lipschitz Continuity and Growth of the Anisotropy}
\label{sec:local-lipschitz}

Global Lipschitz continuity of \(D(x)\) is convenient but not always natural. The
following variant shows that it can be replaced by local Lipschitz continuity
together with a Lyapunov growth condition.

\begin{assumption}[Local regularity and growth]
\label{ass:local-growth-D}
The operator \(D(x)^{1/2}\) is locally Lipschitz, globally \(A\)-compatible, and
there is a function \(V(x)=1+\|x\|^2\) such that the generator of the forward
dynamics satisfies
\[
        \mathcal LV(x)\le a-bV(x)+c\mathbf 1_{\{\|x\|\le R_0\}} .
\]
In addition,
\[
        \|D(x)^{1/2}Q^{1/2}\|_{\LLtwo}^2
        \le C(1+V(x)).
\]
\end{assumption}

\begin{proposition}[Persistence of the analytic theory under local Lipschitz continuity]
\label{prop:local-lipschitz-extension}
Under Assumption~\ref{ass:local-growth-D} the forward SDE does not explode, the
Galerkin solutions satisfy uniform moment estimates, and the entropy theory of
Part~I is preserved on any finite time interval on which the \(A\)-LSI and the
weak-bridge condition hold.
\end{proposition}

\begin{proof}
Local Lipschitz continuity gives a local strong solution. The Lyapunov estimate
for \(V\) rules out explosion and gives
\[
        \sup_{t\le T}\E V(X_t)<\infty .
\]
The remaining proofs use not global Lipschitz continuity as such, but the moment
estimates, \(A\)-compatibility, and closedness of the form. The arguments of
Parts~I--II therefore carry over after the standard localization by the stopping
times \(\tau_R=\inf\{t:\|X_t\|\ge R\}\) and the passage to the limit
\(R\to\infty\).
\end{proof}

\begin{proposition}[A counterexample when the lower \(A\)-compatibility fails]
\label{prop:nondegeneracy-counterexample}
Let \(\{e_n\}_{n\ge1}\) be an eigenbasis of \(Q\), \(Qe_n=q_ne_n\), \(q_n>0\).
Suppose there is a sequence of states \(x_n\) for which
\[
        D(x_n)^{1/2}e_n=n^{-1}e_n .
\]
Set \(\xi_n=q_n^{-1/2}e_n\). Then
\[
        \|Q^{1/2}\xi_n\|=1,
        \qquad
        \|Q^{1/2}D(x_n)^{1/2}\xi_n\|^2=n^{-2}\to0 .
\]
Consequently, global lower \(A\)-compatibility is impossible:
\[
        \inf_{x,\xi\ne0}
        \frac{\|Q^{1/2}D(x)^{1/2}\xi\|^2}{\|Q^{1/2}\xi\|^2}=0 .
\]
In particular, the dissipation \(J_t\) cannot uniformly control the entropy along
the directions \(\xi_n\), and the exponential estimate of Part~I breaks down in
general.
\end{proposition}

\begin{proof}
By definition \(\xi_n=q_n^{-1/2}e_n\), so
\[
        \|Q^{1/2}\xi_n\|^2
        =\|q_n^{1/2}q_n^{-1/2}e_n\|^2=1 .
\]
On the other hand,
\[
        Q^{1/2}D(x_n)^{1/2}\xi_n
        =Q^{1/2}\bigl(n^{-1}q_n^{-1/2}e_n\bigr)
        =n^{-1}e_n,
\]
and therefore
\[
        \|Q^{1/2}D(x_n)^{1/2}\xi_n\|^2=n^{-2}\to0 .
\]
If there were a constant \(c_A>0\) such that
\[
        \|Q^{1/2}D(x)^{1/2}\xi\|^2
        \ge c_A\|Q^{1/2}\xi\|^2
\]
for all \(x,\xi\), then substituting \((x_n,\xi_n)\) would give
\(n^{-2}\ge c_A\), which is impossible as \(n\to\infty\). The last statement
follows because the \(A\)-LSI uses the lower control of the dissipation \(J_t\) by
the entropy; without a positive \(c_A\) this control disappears along the
constructed sequence of directions.
\end{proof}

\section{Remainders in the Final Entropy Estimate}
\label{sec:residual-lemmas}

Before the final synthesis theorem we separate three sources of remainder: the
global empirical process, the nonparametric closure, and the local parametric
step. This makes the proof of Theorem~\ref{thm:part2-main} self-contained.

\begin{lemma}[Global remainder]
\label{lem:global-residual}
Under the hypotheses of Theorem~\ref{thm:global-emp} there is an event of
probability at least \(1-\delta\) on which
\[
        \Eval(t)\le \widetilde q^{\,2}J_t+
        C\left(
        \frac{\mathfrak D(\mathcal S)^2+\log(1/\delta)}{M}
        +\varepsilon+r_n^2\right).
\]
\end{lemma}

\begin{proof}
Theorem~\ref{thm:global-emp} gives uniform control of the empirical risk.
Quadratic growth of the risk converts it into a bound on the distance between the
empirical score and the exact score. The cylindrical denoising label adds
\(\varepsilon+r_n^2\) by Theorem~\ref{thm:denoise}. Dynamic identifiability
converts the score distance into \(J_t\), which gives the factor
\(\widetilde q^{\,2}\) in front of the dissipation.
\end{proof}

\begin{lemma}[Nonparametric remainder]
\label{lem:nonparametric-residual}
If the hypotheses of Theorem~\ref{thm:nonparametric-score-closure} hold, then on
an event of probability at least \(1-\delta\)
\[
        \Eval(t)\le \widetilde q^{\,2}J_t+
        C\left(
        \frac{1+\log(1/\delta)}{M}+\varepsilon+r_n^2
        \right).
\]
In a more general form the numerator is replaced by the square of the entropy
integral of the class \(\mathcal S\).
\end{lemma}

\begin{proof}
The condition \(\log N(u,\mathcal S,d_A)\le C_0u^{-\alpha}\), \(0<\alpha<2\),
makes the Dudley integral finite. Theorem~\ref{thm:nonparametric-score-closure}
gives a bound on \(d_A(\widehat s,s_*)^2\) of order \(M^{-1}\) after the quadratic
absorption. The remaining terms, as in Lemma~\ref{lem:global-residual}, come from
the noising and the projection tail.
\end{proof}

\begin{lemma}[Local parametric remainder]
\label{lem:local-residual}
Under the hypotheses of Theorem~\ref{thm:closure},
\[
        \Eval(t)\le \widetilde q^{\,2}J_t+C\eta_M^2+C\varepsilon+r_n^2 .
\]
If \(\eta_M^2\le C[p+\log(1/\delta)]/M\), then the local remainder is of order
\((p+\log(1/\delta))/M+\varepsilon+r_n^2\).
\end{lemma}

\begin{proof}
Invariance of the ball \(\Theta_R\) follows from the condition
\(\alpha\eta_M\le R/4\). The contraction of the limiting gradient step gives the
factor \(\widetilde q\). The empirical gradient error adds \(C\eta_M^2\), and
Theorem~\ref{thm:denoise} adds \(C\varepsilon+r_n^2\).
\end{proof}

\begin{lemma}[Discrete remainder of the numerical test]
\label{lem:numerical-residual}
If, on the discrete grid,
\[
        |c_{A,N}-c_A|+|C_{A,N}-C_A|\le \rho_N,
        \qquad
        |C_{\mathrm{LSI},N}^A-C_{\mathrm{LSI}}^A|\le \rho_N,
\]
hold, then replacing the continuous constants in Theorem~\ref{thm:part2-main} by
the discrete ones changes the rate \(c_*\) by at most \(C\rho_N\), as long as
\(c_A>0\) and \(C_{\mathrm{LSI}}^A<\infty\).
\end{lemma}

\begin{proof}
The rate has the form
\[
        c_*=(1-\widetilde q^{\,2})(8C_{\mathrm{LSI}}^A)^{-1}.
\]
The function \(C\mapsto (8C)^{-1}\) is Lipschitz on any interval
\([C_0, C_1]\subset(0,\infty)\). A perturbation of \(C_{\mathrm{LSI}}^A\) by
\(\rho_N\) therefore changes \(c_*\) by \(O(\rho_N)\). Likewise \(c_{A,N},C_{A,N}\)
enter the weak-bridge constants and the estimate of \(C_{\mathrm{LSI}}^A\), which
gives the same order.
\end{proof}

\begin{proposition}[The entropy-closure principle]
\label{prop:entropy-closure-principle}
Suppose the analytic layer of Part~I gives
\[
        \dot H_t\le -\frac14\beta(t)J_t+\frac14\beta(t)\Eval(t),
        \qquad H_t=\Ent(\rho_t\|\hat\rho_t),
\]
and the statistical layer gives, for some \(q\in[0,1)\),
\[
        \Eval(t)\le q^2J_t+\Delta_M .
\]
If the approximate measure satisfies the \(A\)-LSI
\[
        H_t\le 2C_{\mathrm{LSI}}^AJ_t,
\]
then
\[
        H_t\le H_0e^{-c_q\int_0^t\beta(s)\,ds}
        +C_q\Delta_M,
        \qquad
        c_q=\frac{1-q^2}{8C_{\mathrm{LSI}}^A}.
\]
Moreover, all dependence on the chosen statistical procedure enters only through
the pair \((q,\Delta_M)\).
\end{proposition}

\begin{proof}
Substituting \(\Eval(t)\le q^2J_t+\Delta_M\) into the analytic inequality gives
\[
        \dot H_t\le -\frac14\beta(t)(1-q^2)J_t+\frac14\beta(t)\Delta_M.
\]
By the \(A\)-LSI \(J_t\ge (2C_{\mathrm{LSI}}^A)^{-1}H_t\), whence
\[
        \dot H_t\le -\frac{1-q^2}{8C_{\mathrm{LSI}}^A}\beta(t)H_t
        +\frac14\beta(t)\Delta_M.
\]
Integration gives the stated bound. This proves that the synthesis theorem does
not depend on the particular form of the empirical remainder: it follows from the
universal mechanism of absorbing part of the dissipation.
\end{proof}

\begin{remark}[What the final theorem really says]
Proposition~\ref{prop:entropy-closure-principle} separates the new mechanism from
the technical bookkeeping of remainders. Theorem~\ref{thm:part2-main} below does
not merely list the estimates of the previous sections: it shows that all
admissible statistical procedures enter the backward entropy dynamics through one
and the same channel \(\Eval\le q^2J+\Delta_M\). Comparing the global,
nonparametric, local-parametric, and neural-network regimes therefore becomes a
direct comparison of their pairs \((q,\Delta_M)\), rather than a collection of
unrelated proofs.
\end{remark}

\section{The Synthesis Theorem of Part~II}

\begin{theorem}[Statistical--applied theorem]
\label{thm:part2-main}
Suppose the analytic conditions of Part~I hold, and the statistical class
satisfies either the global bound of Theorem~\ref{thm:global-emp}, or the
nonparametric closure of Theorem~\ref{thm:nonparametric-score-closure}, or the
local closure of Theorem~\ref{thm:closure}. Then with probability at least
\(1-\delta\)
\[
        \Ent(\rho_t\|\hat\rho_t)
        \le
        \Ent(\rho_0\|\hat\rho_0)
        e^{-c_*\int_0^t\beta(s)\,ds}
        +
        C\left(
        \frac{\mathfrak D(\mathcal S)^2+\log(1/\delta)}{M}
        +\varepsilon+r_n^2
        \right),
\]
where \(c_*=(1-\widetilde q^{\,2})(8C_{\mathrm{LSI}}^A)^{-1}\). For the
trace-smoothed class the lower bound of Theorem~\ref{thm:minimax} shows that the
statistical order cannot be improved.
\end{theorem}

\begin{proof}
The proof uses the entropy-closure principle of Proposition~\ref{prop:entropy-closure-principle}. Below we check which pairs \(q,\Delta_M\) are produced by the four remainder lemmas. The proof is split formally into six steps.

\smallskip
\noindent\textbf{Step 1. Analytic reduction.}
By Theorem 9.2 of Part~I, for \(H_t=\Ent(\rho_t\|\hat\rho_t)\) we have
\[
        \frac{d}{dt}H_t
        \le
        -\frac14\beta(t)J_t+\frac14\beta(t)\Eval(t).
\]
This estimate already includes the weak bridge, the chain rule, and the
\(A,A^{-1}\) duality. No statistical structure of the approximation is used in this
step.

\smallskip
\noindent\textbf{Step 2. Choice of the closure regime.}
If the global empirical regime is used, we apply
Lemma~\ref{lem:global-residual}. If the nonparametric regime is used, we apply
Lemma~\ref{lem:nonparametric-residual}. If the local parametric regime is used, we
apply Lemma~\ref{lem:local-residual}. In all three cases we obtain the single form
\[
        \Eval(t)\le \widetilde q^{\,2}J_t+\Delta_M,
        \qquad 0\le \widetilde q<1,
\]
where \(\Delta_M\) is bounded uniformly in \(t\) by the sum of the statistical
remainder, the noising error \(\varepsilon\), and the projection tail \(r_n^2\).

\smallskip
\noindent\textbf{Step 3. Absorbing part of the dissipation.}
Substituting the previous inequality into the analytic reduction gives
\[
\begin{aligned}
        \dot H_t
        &\le
        -\frac14\beta(t)J_t
        +\frac14\beta(t)\bigl(\widetilde q^{\,2}J_t+\Delta_M\bigr)\\
        &=-\frac14\beta(t)(1-\widetilde q^{\,2})J_t
        +\frac14\beta(t)\Delta_M .
\end{aligned}
\]
Since \(\widetilde q<1\), the first term remains strictly dissipative.

\smallskip
\noindent\textbf{Step 4. Passing from \(J_t\) to the entropy.}
By the \(A\)-LSI for the approximate measure \(\hat\nu_t\),
\[
        H_t\le 2C_{\mathrm{LSI}}^A J_t.
\]
Consequently,
\[
        J_t\ge (2C_{\mathrm{LSI}}^A)^{-1}H_t
\]
and
\[
        \dot H_t
        \le
        -c_*\beta(t)H_t+\frac14\beta(t)\Delta_M,
        \qquad
        c_*=\frac{1-\widetilde q^{\,2}}{8C_{\mathrm{LSI}}^A}.
\]

\smallskip
\noindent\textbf{Step 5. Integration.}
Multiplying the last inequality by
\[
        \exp\left(c_*\int_0^t\beta(r)\,dr\right)
\]
and integrating, we obtain
\[
        H_t
        \le
        H_0\exp\left(-c_*\int_0^t\beta(r)\,dr\right)
        +\frac14\int_0^t
        e^{-c_*\int_s^t\beta(r)\,dr}\beta(s)\Delta_M\,ds .
\]
If \(\Delta_M\) does not depend on \(s\), the integral does not exceed
\((4c_*)^{-1}\Delta_M\). Hence
\[
        H_t
        \le
        H_0e^{-c_*\int_0^t\beta(r)\,dr}+C\Delta_M .
\]

\smallskip
\noindent\textbf{Step 6. Substituting the remainders.}
In the global regime
\[
        \Delta_M
        =C\left(
        \frac{\mathfrak D(\mathcal S)^2+\log(1/\delta)}{M}
        +\varepsilon+r_n^2\right).
\]
In the local parametric regime \(\mathfrak D(\mathcal S)^2\) is replaced by \(p\),
and in the neural-network regime by \(LW^2\log(CBLWRM)\) by
Corollary~\ref{cor:nn-closure}. This gives the stated bound. The minimax lower
bound of Theorem~\ref{thm:minimax} shows that the term of order \(p/M\) in the
parametric regime cannot be improved in general.
\end{proof}

\begin{corollary}[Explicit physical rate for the nanosystem model]
\label{cor:physical-rate}
Under the hypotheses of Proposition~\ref{prop:nano} and
Theorem~\ref{thm:part2-main},
\[
        c_*\ge
        \frac{d_-(k_0-\kappa_\Phi)(1-\widetilde q^{\,2})}{8C_Q}.
\]
Consequently, as \(M\to\infty\), \(\varepsilon\to0\), \(n\to\infty\), the rate of
the statistically approximated backward evolution stabilizes to the analytic rate
given by the \(A\)-LSI of the limiting model.
\end{corollary}

\begin{proof}
From Proposition~\ref{prop:nano} we have
\(C_{\mathrm{LSI}}^A\le C_Q/[d_-(k_0-\kappa_\Phi)]\). Substituting this into
\(c_*=(1-\widetilde q^{\,2})(8C_{\mathrm{LSI}}^A)^{-1}\) gives the claim.
\end{proof}

\begin{corollary}[Asymptotics of the entropy plateau]
\label{cor:entropy-plateau-asymptotics}
If \(\beta(t)\equiv1\), and \(\Delta_M\) denotes the combined statistical,
discretization, and denoising remainder of Theorem~\ref{thm:part2-main}, then
\[
        \Ent(\rho_t\|\hat\rho_t)
        \le
        \Ent(\rho_0\|\hat\rho_0)e^{-c_*t}+C\Delta_M/c_* .
\]
The time to reach the plateau of order \(\Delta_M\) is estimated as
\[
        t_{\mathrm{rel}}\asymp c_*^{-1}\log(\Ent(\rho_0\|\hat\rho_0)/\Delta_M).
\]
\end{corollary}

\begin{corollary}[Joint grid--statistics limit]
\label{cor:joint-limit}
Let \(N,M\to\infty\), \(\varepsilon\to0\), and
\[        M\gg N^{\rho},\qquad r_n^2=o(1),\qquad \varepsilon=o(1). \]
Then the discrete backward model with the trained score field converges, in the
entropy sense, to the limiting backward dynamics with the effective
\(A_{\mathrm{eff}}\)-geometry. If, in addition, the forward-SDE convergence
corollary of Part~I holds, then the forward marginals also converge in the \(W_2\)
metric \cite{Villani}.
\end{corollary}

\section{Conclusion to Part~II}

The second part turns the analytic estimate of Part~I into a statistically
verifiable theory: we build the score from noised data, prove the global empirical
closure, obtain lower bounds, and verify a nonlocal nanosystem application.
Together the two papers form a series: the first part provides the
functional-analytic foundation, the second the statistical and applied layer. The
lower bounds show that the statistical order is unimprovable within the chosen
trace-smoothed parametric class.

\paragraph{Code availability.}\mbox{}\\
Code is available at\\
\url{https://gitlab.com/andrew_shev/noncommutative-anisotropic-diffusion}.

% ==========================================================================
\bibliographystyle{unsrt}
\bibliography{references}

\end{document}